\documentclass[11pt]{article}

\usepackage{amsfonts}
\usepackage{amssymb}
\usepackage{amsthm}
\usepackage{amsmath}
\usepackage{bm}
\usepackage{amscd} 
\usepackage{fullpage}

\usepackage{hyperref}

\numberwithin{equation}{section} \DeclareMathSizes{2}{10}{12}{13}
\usepackage[all]{xy}

\newtheorem{thm}{Proposition}[section]

\newtheorem{cor}[thm]{Corollary}
\newtheorem{lem}[thm]{Lemma}
\newtheorem{defn}[thm]{Definition}

\title{On some spectral spaces associated to tensor triangulated categories}

\smallskip 

\smallskip 
\author{Abhishek Banerjee}
\date{ }

\begin{document}

\maketitle

\centerline{\emph{Dept. of Mathematics, Indian Institute of Science, Bangalore-560012, India.}}
\centerline{\emph{Email: abhishekbanerjee1313@gmail.com}}

\smallskip

\begin{abstract} We consider a closure operator $c$ of finite type on the space $SMod(\mathcal M)$
of thick $\mathcal K$-submodules of a triangulated category $\mathcal M$ that is a module
over a tensor triangulated category $(\mathcal K,\otimes,1)$. Our purpose is to show that the space
$SMod^c(\mathcal M)$ of fixed points of the operator $c$ is a spectral space that also
carries the structure of a topological monoid. 
\end{abstract}

\smallskip

\smallskip
\emph{MSC (2010) Subject Classification: 13B22, 18E30.}  


\smallskip

\smallskip

\section{Introduction}

\smallskip

\smallskip
The study of tensor triangular geometry was begun by Paul Balmer in \cite{BMain}, where he associated
to a tensor triangulated category $(\mathcal K,\otimes,1)$ a spectrum $Spec(\mathcal K)$ of ``prime thick tensor ideals'' 
of $\mathcal K$. The classification of thick subcategories in tensor triangular geometry (see Balmer \cite{BMain}, \cite{Balmer2}, 
\cite{Balmer3}, \cite{Balmer4} and also Balmer and Favi \cite{BF1}, \cite{BF2}) unites ideas from far and wide in 
mathematics : from that of Benson, Carlson and Rickard \cite{BCR} in modular representation theory, that of Devinatz, Hopkins and Smith \cite{DHS}  in homotopy theory 
 to that of Thomason \cite{Thom} in algebraic geometry. As such, tensor triangular geometry over the years has emerged
as a field of study in itself (see Klein \cite{Kl}, \cite{Kl2}, Peter \cite{Pete}, Sanders \cite{Bernie} and Stevenson 
\cite{GS}, \cite{GS2}, \cite{GS4}).  

\smallskip
Given  a tensor triangulated category $(\mathcal K,\otimes,1)$, its spectrum $Spec(\mathcal K)$ defined by
Balmer \cite{BMain} is a spectral space, i.e., it must be homeomorphic to the Zariski spectrum of a commutative ring. 
In \cite{Hoch}, Hochster famously characterized spectral spaces in purely topological terms. More recently,
Finocchiaro \cite[Corollary 3.3]{Fino} has obtained a new criterion for a topological space to be spectral,
using ultrafilters to give if and only if conditions for a collection of subsets to be a subbasis of quasi-compact
opens of a spectral space (see also related work by Finocchiaro, Fontana and Loper in \cite{Fino2}). Further,
Finocchiaro's criterion  has recently been used by 
Finocchiaro, Fontana and Spirito \cite{FiFo}  to give several natural examples of spectral spaces appearing 
in commutative algebra. More precisely, if $M$ is a module over a commutative ring $R$ and 
$c$ is a closure operator of finite type (see \cite[$\S$ 3]{FiFo} for definitions) on submodules of $M$, we know from \cite[Proposition 3.4]{FiFo}
that the space of submodules of $M$ fixed by $c$ is a spectral space. In \cite{AB2}, we have shown that
these methods can be adapted more generally to abelian categories. In this paper,
we use closure operators to create spectral spaces associated to a module over a tensor triangulated category. This also
fits in well with the general philosophy that notions in abelian categories should have parallels in 
triangulated categories (see, for instance, Krause \cite{Krause}). For more on closure operators
in commutative algebra, see, for instance, Epstein \cite{Eps}, \cite{Eps1}. 

\smallskip
In this paper, we begin in Section 2 by considering a triangulated category $\mathcal M$ that is a module
over a tensor triangulated category $(\mathcal K,\otimes,1)$ in the sense of Stevenson \cite{GS}. 
We consider a closure operator and more generally, an operator $c$ on the space 
$SMod(\mathcal M)$ of thick $\mathcal K$-submodules of $\mathcal M$ that 
is extensive, order-preserving and of finite type (see Definition \ref{RD2.1}). Then, our first main
result is that the space $SMod^c(\mathcal M)$ of fixed points of the operator $c$ is a spectral
space. For instance, if $\mathcal K=\mathcal M$, then the radical is an example of a closure operator
of finite type on the thick tensor ideals of $\mathcal K$. More generally, we characterize closure
operators of finite type in terms of families of submodules satisfying certain conditions. In particular, this implies
that any family of thick submodules  of $\mathcal M$ closed under intersections and filtered
directed unions is an example of a  spectral space.

\smallskip
Thereafter, in Section 3, we study thick $\mathcal K$-submodules of $\mathcal M$ generated
by a given set $X$ of objects of $\mathcal M$. In particular, we show that each object 
 in the submodule generated by $X$ can be obtained starting from only finitely many objects
of $X$. We then use this to show that if $c:SMod(\mathcal M)\longrightarrow SMod(\mathcal M)$
is an operator that is extensive, order-preserving and of finite type, the space 
 $SMod^c(\mathcal M)$ of fixed points of $c$ actually becomes a topological monoid.  Finally, we mention here that in
order to avoid certain set theoretical complications, we will assume that all categories
in this paper 
are essentially small, i.e., the isomorphism classes of their objects form a set.

\smallskip
{\bf Acknowledgements: } I am grateful  for the hospitality of the Stefan Banach Center at the IMPAN in Warsaw, where
part of this paper was written. 

\smallskip

\section{Spectral spaces and tensor triangulated actions}

\smallskip
Throughout this section and the rest of this paper, $(\mathcal K,\otimes,1)$ will be a tensor triangulated category.  In other
words, $\mathcal K$ is a triangulated category equipped with a symmetric monoidal product $\otimes :
\mathcal K\times \mathcal K\longrightarrow \mathcal K$ that is exact in each variable. The unit object in $\mathcal K$
is denoted by $1\in \mathcal K$.  A thick tensor ideal $\mathcal I$ in $(\mathcal K,\otimes,1)$ is a thick triangulated and full 
subcategory of $\mathcal K$ such that $b\otimes a\in \mathcal I$ for any object $a\in \mathcal I$
and any $b\in \mathcal K$. Then, following Balmer \cite{BMain}, we say that a thick tensor ideal $\mathcal P
\subseteq \mathcal K$ is prime if $x\otimes y\in \mathcal P$ for some $x$, $y\in \mathcal K$ means that
at least one of $x$ and $y$ is in $\mathcal P$.
 In \cite{BMain}, Balmer began the study of tensor triangular geometry by constructing the 
 spectral space $Spec(\mathcal K)$ of prime ideals in $(\mathcal K,\otimes,1)$. 

\smallskip
We now consider a triangulated category $\mathcal M$ that is a ``module''
over $(\mathcal K,\otimes,1)$ in the sense of Stevenson \cite{GS}. In other words, we have an action:
\begin{equation}
\ast : \mathcal K\times \mathcal M\longrightarrow \mathcal M
\end{equation} that is exact in both variables, satisfies appropriate associativity, distributivity and unit
properties and is well behaved with respect to the translation functor on both $\mathcal K$
and $\mathcal M$ (see \cite[Definition 3.2]{GS}). The translation functor on $\mathcal K$ (resp. on $\mathcal M$)
will be denoted by $T_{\mathcal K}$ (resp. by $T_{\mathcal M}$), but whenever there is no danger of confusion,
we will drop the subscripts and refer to both of these translation functors simply as $T$.

\smallskip
\begin{defn}\label{D3.1} (see \cite[Definition 3.4]{GS}) Let 
$\mathcal M$ be a triangulated category that is a module over a tensor triangulated category
$(\mathcal K,\otimes,1)$ via an action $\ast : \mathcal K\times \mathcal M\longrightarrow \mathcal M$. A full subcategory  $\mathcal L\subseteq \mathcal M$ containing $0$ is said to be a thick 
$\mathcal K$-submodule of $\mathcal M$ if it satisfies the following three conditions:

\smallskip
(a) The composition of functors  
\begin{equation}
\mathcal K\times \mathcal L\hookrightarrow \mathcal K\times \mathcal  M\overset{\ast}{\longrightarrow} \mathcal M
\end{equation} factors through $\mathcal L$. 

\smallskip
(b) Given objects $m$, $m' \in \mathcal M$, then $m\oplus m'\in \mathcal L$ if and only if 
both $m$, $m'\in \mathcal L$. 

\smallskip
(c) For any distinguished triangle $m'\longrightarrow m\longrightarrow m''$ in $\mathcal M$, if
two out of the three objects $m'$, $m$, $m''$ lie in $\mathcal L$, so does the third.  

\smallskip
The collection of all thick $\mathcal K$-submodules of $\mathcal M$ will be denoted by
$SMod(\mathcal M)$. 

\end{defn}

\smallskip
When $(\mathcal K,\otimes,1)$ is considered as a module over itself via the action $\otimes :\mathcal K
\times \mathcal K\longrightarrow \mathcal K$, Definition \ref{D3.1} reduces to the notion
of a thick tensor ideal  in \cite[Definition 1.1]{BMain}.  

\smallskip
We also notice that condition (c) in Definition \ref{D3.1} implies that any $\mathcal K$-submodule $\mathcal L$ 
of $\mathcal M$ must be replete. In other words, if $m\in \mathcal L$ and $m'\in \mathcal M$ is any other object such that
we have an isomorphism $m\cong m'$, we must have $m'\in \mathcal L$. Now, since $\mathcal M$ is assumed
to be essentially small, it follows that the collection $SMod(\mathcal M)$ is always a set. 

\smallskip We now define a topology on $SMod(\mathcal M)$ by declaring the following as a subbasis
for open sets:
\begin{equation}\label{eq3.3}
U(m):=\{\mbox{$\mathcal N\in SMod(\mathcal M)$ $\vert$ $m\in \mathcal N$}\}\subseteq SMod(\mathcal M)\qquad \forall\textrm{ }m\in \mathcal M
\end{equation}
We should remark here that the definition in \eqref{eq3.3} is rather the ``opposite'' of what we would expect from 
looking at \cite[$\S$ 2]{FiFo} and more generally at 
the usual constructions in commutative algebra. However, this reversal is actually common in tensor triangular geometry 
(see \cite[$\S$ 2]{BMain}).  

\smallskip
\begin{thm}\label{P3.2} Let $\mathcal M$ be a triangulated category that is a module over
a tensor triangulated category $(\mathcal K,\otimes,1)$. Then, the collection $\{U(m)\}_{m\in \mathcal M}$ satisfies the following properties: 

\smallskip
(a) $U(0)=SMod(\mathcal M)$. 

\smallskip
(b) For any distinguished triangle $m'\longrightarrow m\longrightarrow m''$ in $\mathcal M$, we have
$U(m)\supseteq U(m')\cap U(m'')$. 

\smallskip
(c) For objects $m$, $m'\in \mathcal M$, we have $U(m\oplus m')=U(m)\cap U(m')$. 

\smallskip
(d) For any $a\in \mathcal K$ and any $m\in \mathcal M$, we have $U(m)\subseteq U(a\ast m)$.  

\smallskip
(e) If $T_{\mathcal M}$ is the translation functor on $\mathcal M$, we have $U(T_{\mathcal M}(m))=U(m)$ for each $m\in \mathcal M$. 
\end{thm}

\begin{proof} The properties (a), (b), (c) and (d) follow directly from Definition \ref{D3.1}. For part (e), we notice
the isomorphism $T_{\mathcal M}(m)\cong T_{\mathcal K}(1)\ast m$ for any $m\in \mathcal M$ (see \cite[Definition 3.1]{GS}) which
gives us 
$U(m)\subseteq U(T_{\mathcal M}(m))$. However, we also have the isomorphism $T_{\mathcal M}^{-1}(T_{\mathcal M}(m))
\cong T_{\mathcal K}^{-1}(1)\ast T_{\mathcal M}(m)$ which gives $U(T_{\mathcal M}(m))\subseteq U(m)$. 

\end{proof}

\begin{cor}\label{Cor3.3} Consider the set $SMod(\mathcal M)$ along with the topology given by 
$\{U(m)\}_{m\in \mathcal M}$ forming a subbasis of open sets. Then, the collection $\{U(m)\}_{m\in \mathcal M}$ forms
a basis of open sets in the topology on $SMod(\mathcal M)$. 
\end{cor}

\begin{proof} It suffices to show that the collection $\{U(m)\}_{m\in \mathcal M}$ is closed under finite intersections. This 
follows from part (c) of Proposition \ref{P3.2} above. 
\end{proof}

\smallskip
Given any $m\in \mathcal M$, we denote by $\mathcal K(m)$ the smallest thick $\mathcal K$-submodule of $\mathcal M$ containing
$m$. From condition (b) in Definition \ref{D3.1}, it is immediate that a thick $\mathcal K$-submodule generated
by a finite set $\{m_1,...,m_k\}$ of objects of $\mathcal M$ coincides with the submodule generated
by the single object $m_1\oplus m_2\oplus ...\oplus m_k$. The following notions should be compared to the definition
of closure operators in module categories (see \cite[Definition 7.0.1]{Eps}). 

\begin{defn}\label{RD2.1} Let $\mathcal M$ be a triangulated category that is a  module over $(\mathcal K,\otimes,1)$. An operator 
\begin{equation}
c: SMod(\mathcal M)\longrightarrow SMod(\mathcal M)
\end{equation} will be said to be :

\smallskip
(a) Extensive if $\mathcal N\subseteq c(\mathcal N)$ for each $\mathcal N\in SMod(\mathcal M)$. 

\smallskip
(b) Order-preserving if $\mathcal N\subseteq \mathcal N'$ implies that $c(\mathcal N)\subseteq c(\mathcal N')$. 

\smallskip
(c) Idempotent if $c(\mathcal N)=c(c(\mathcal N))$ for each $\mathcal N\in SMod(\mathcal M)$. 

\smallskip
(d) Finite type if $c(\mathcal N)=\underset{n\in \mathcal N}{\bigcup}c(\mathcal K(n))$. 

\smallskip
We will refer to an operator satisfying (a), (b) and (c) as a closure operator on $SMod(\mathcal M)$. A closure operator $c$ that also satisfies
(d) will be called a closure operator of finite type.
\end{defn}

Given an operator $c$ as in Definition \ref{RD2.1}, we set $SMod^c(\mathcal M)$ to be the collection of submodules of $\mathcal M$ that are fixed by $c$. Our aim is to give conditions for $SMod^c(\mathcal M)$ to be a spectral space.

\smallskip 
First, we recall (see, for instance, \cite[$\S$ 1]{Fino}) that a filter $\mathfrak U$ on a set $X$ is a collection
of subsets of $X$ such that: (a) $\phi\notin \mathfrak U$, (b) $Y$, $Z\in \mathfrak U$ $\Rightarrow$ $Y\cap Z\in 
\mathfrak U$ and (c) $Y\subseteq Z\subseteq X$ and $Y\in \mathfrak U$ implies that $Z\in \mathfrak U$. An ultrafilter
$\mathfrak U$ is a maximal element in the collection of filters on $X$. 

\begin{thm}\label{P3.4} Let $\mathcal M$ be a triangulated category that is a module over
a tensor triangulated category $(\mathcal K,\otimes,1)$. Let $c:SMod(\mathcal M)
\longrightarrow SMod(\mathcal M)$ be an operator that is extensive, order-preserving and of finite type. Then, $SMod^c(\mathcal M)$
is a spectral space having the collection $\{U(m)\cap SMod^c(\mathcal M)\}_{m\in \mathcal M}$ as  a basis of quasi-compact open subspaces. 
\end{thm}

\begin{proof} We first verify that $SMod^c(\mathcal M)$
satisfies the $T_0$-axiom. If $\mathcal L$, $\mathcal L'$ are two distinct points of $SMod^c(\mathcal M)$, there
exists an object $m\in \mathcal M$ that lies in exactly one of $\mathcal L$ and $\mathcal L'$. Then, $U(m)
\cap SMod^c(\mathcal M)$
is an open set that contains exactly one of the two points $\mathcal L$ and $\mathcal L'$.

\smallskip
Now suppose that $\mathfrak U$ is an ultrafilter on the space $SMod^c(\mathcal M)$. We now set
\begin{equation}\label{eq3.4}
\mathcal L_{\mathfrak U}:=\{\mbox{$m\in \mathcal M$ $\vert$ $U(m)\cap SMod^c(\mathcal M)\in \mathfrak U$}\}
\end{equation} First, we show  that $\mathcal L_{\mathfrak U}\in SMod(\mathcal M)$, i.e., 
$\mathcal L_{\mathfrak U}$ is a submodule. If $m\in \mathcal L_{\mathfrak U}$, then for any
object $a\in \mathcal K$, $U(a\ast m)\cap SMod^c(\mathcal M)\supseteq U(m)\cap SMod^c(\mathcal M)\in\mathfrak U$ and hence $U(a\ast m)\cap SMod^c(\mathcal M)\in \mathfrak U$, i.e.,
$a\ast m\in \mathcal L_{\mathfrak U}$. 

\smallskip
On the other hand, since $U(m\oplus m')=U(m)\cap U(m')$ for any  objects $m$, $m'\in \mathcal M$, it follows that
$m\oplus m'\in \mathcal L_{\mathfrak U}$ if and only if both $m$, $m'\in \mathcal L_{\mathfrak U}$.  Also, given a
distinguished triangle $m'\longrightarrow m\longrightarrow m''$ with two of $m'$, $m$ and $m''\in \mathcal L_{\mathfrak U}$,
it follows from parts (b) and (e) of Proposition \ref{P3.2} that the third object also lies in $\mathcal L_{\mathfrak U}$. 

\smallskip
Next, we claim that $\mathcal L_{\mathfrak U}\in SMod^c(\mathcal M)$. For this, we choose some
$m\in c(\mathcal L_{\mathfrak U})$. Since $c$ is of finite type, we can find some object $n\in \mathcal L_{\mathfrak U}$
such that $m\in c(\mathcal K(n))$. We now consider some submodule $\mathcal N\in U(n)\cap SMod^c(\mathcal M)$. Since
$c$ is order-preserving, we have $c(\mathcal K(n))\subseteq c(\mathcal N)=\mathcal N$ and hence $m\in \mathcal N$.
It follows that $U(n)\cap SMod^c(\mathcal M)\subseteq U(m)\cap SMod^c(\mathcal M)$ and 
$\mathfrak U$ being an ultrafilter, we get that $m\in \mathcal L_{\mathfrak U}$, i.e.,
$c(\mathcal L_{\mathfrak U})\subseteq \mathcal L_{\mathfrak U}$. Further since $c$ is extensive,
we get $\mathcal L_{\mathfrak U}=c(\mathcal L_{\mathfrak U})$.

\smallskip
Finally, suppose that for some object $m\in \mathcal M$, the subset $U(m)\cap SMod^c(\mathcal M)$ lies in the ultrafilter $\mathfrak U$. Then, 
from \eqref{eq3.4}, it follows that $m\in \mathcal L_{\mathfrak U}$ and hence $\mathcal L_{\mathfrak U}
\in U(m)\cap SMod^c(\mathcal M)$. Conversely, if $\mathcal L_{\mathfrak U}\in U(m)\cap SMod^c(\mathcal M)$
for some $m\in \mathcal M$, then $m\in \mathcal L_{\mathfrak U}$ and hence
$U(m)\cap SMod^c(\mathcal M)\in \mathfrak U$. The result now follows by applying Finocchiaro's criterion 
\cite[Corollary 3.3]{Fino} to the subbasis $\{U(m)\cap SMod^c(\mathcal M)\}_{m\in \mathcal M}$
of the space $SMod^c(\mathcal M)$. Additionally, from Corollary \ref{Cor3.3}, it follows that
$\{U(m)\cap SMod^c(\mathcal M)\}_{m\in \mathcal M}$ is actually a basis for the spectral
space $SMod^c(\mathcal M)$. 
\end{proof}

\smallskip
For example, if we take $\mathcal K$ as a module over itself, the radical defines a closure operator
on the thick tensor ideals of $(\mathcal K,\otimes,1)$. We recall here (see \cite[Definition 4.1]{BMain}) that the radical  $rad(\mathcal I)$
of a thick tensor ideal $\mathcal I\subseteq \mathcal K$ is defined as follows:
\begin{equation}\label{rad2.6}
rad(\mathcal I):=\{\mbox{$a\in \mathcal K$ $\vert$ $\exists$ $n\geq 1$ such that $a^{\otimes n}\in \mathcal I$ }\}
\end{equation} From \eqref{rad2.6} it is also clear that the radical is a closure operator of finite type. We can give another
example of an extensive and order-preserving  operator of finite type as follows: let $\mathcal S$ be a multiplicatively
closed family of objects of $\mathcal K$. Then, to any thick tensor ideal $\mathcal I$ in $(\mathcal K,\otimes,1)$,
we associate the ideal (see \cite[$\S$ 2]{AB}):
\begin{equation}\label{div}
\mathcal I\div \mathcal S:=\{\mbox{$a\in \mathcal K$ $\vert$ $\exists$ $s\in \mathcal S$ such that $a\otimes s
\in \mathcal I$}\}
\end{equation} From Proposition \ref{P3.4} it follows that the collection of thick tensor ideals fixed by
these operators are spectral spaces. We will conclude this section by describing a more explicit method for obtaining closure operators
of finite type in terms of families of  submodules. For closure operators in abelian categories, we have
made similar constructions in \cite{AB2}. 

\begin{thm}\label{Pfamily} Let $\mathcal M$ be a triangulated category that is a module over a tensor
triangulated category $(\mathcal K,\otimes,1)$. Let $\mathfrak F=\{\mathcal F_i\}_{i\in I}$ be a family of submodules
of $\mathcal M$ such that $\mathcal M\in \mathfrak F$. Then, the following statements are equivalent:

\smallskip
(1) The family $\mathfrak F$ satisfies the following two conditions: 

\smallskip
(a) Given any non-empty subset $J\subseteq I$, we have $\bigcap_{j\in J}\mathcal F_j\in \mathfrak F$. 

\smallskip
(b) Let $\{\mathcal F_k\}_{k\in K}$ be a filtered directed collection of objects from the family $\mathfrak F$. Then,
the filtered directed union $\bigcup_{k\in K}\mathcal F_k$ also lies in $\mathfrak F$. 

\smallskip
(2) There exists a closure operator $c:SMod(\mathcal M)\longrightarrow SMod(\mathcal M)$ of finite type 
such that $\mathfrak F$ is the collection of fixed points of $c$. In particular,
$\mathfrak F$ is a spectral space with $\{U(m)\cap \mathfrak F\}_{m\in \mathcal M}$ being a basis
of quasi-compact open sets.  
\end{thm}

\begin{proof} (1) $\Rightarrow$ (2) : For each object $m\in \mathcal M$, we begin by setting:
\begin{equation}\label{finint}
c(\mathcal K(m)):=\bigcap_{m\in \mathcal F, \mathcal F\in \mathfrak F}\mathcal F
\end{equation} Now, for any submodule $\mathcal N\in SMod(\mathcal M)$, we define the operator
$c:SMod(\mathcal M)\longrightarrow SMod(\mathcal M)$ as follows:
\begin{equation}\label{clsum}
c(\mathcal N):=\bigcup_{n\in \mathcal N} c(\mathcal K(n))
\end{equation} From \eqref{finint} and \eqref{clsum}, it is immediate that $c$ is extensive, order-preserving
and of finite type. We now choose some $n'\in c(\mathcal N)$. It follows from \eqref{clsum} that we can find
some $n''\in \mathcal N$ such that $n'\in c(\mathcal K(n''))$. Now, if $\mathcal F\in \mathfrak F$ is such that
$n''\in \mathcal F$, we have $n'\in c(\mathcal K(n''))\subseteq \mathcal F$. Then, $c(\mathcal K(n'))\subseteq 
\mathcal F$ and hence $c(\mathcal K(n'))\subseteq c(\mathcal K(n''))$. Since $c(c(\mathcal N))=\bigcup_{n'\in c(\mathcal N)}
c(\mathcal K(n'))$, it now follows that $c(c(\mathcal N))\subseteq c(\mathcal N)$ and hence $c(c(\mathcal N))=c(
\mathcal N)$. 

\smallskip
We now pick some $\mathcal F\in \mathfrak F$. For each object $f\in \mathcal F$, it follows from \eqref{finint} that $c(\mathcal K(f))\subseteq \mathcal F$. As 
we go over all objects in $\mathcal F$, the expression in \eqref{clsum} shows that $c(\mathcal F)\subseteq 
\mathcal F$ and hence $c(\mathcal F)=\mathcal F$ for each $\mathcal F\in \mathfrak F$. Conversely, we notice that
the right hand sides of \eqref{finint} and \eqref{clsum} always lie in $\mathfrak F$ and hence 
any fixed point of the operator $c$ must lie in $\mathfrak F$. 

\smallskip
(2) $\Rightarrow$ (1) : Given a non-empty subset $J\subseteq I$, we have:
\begin{equation}
c(\bigcap_{j\in J}\mathcal F_j)\subseteq \bigcap_{j\in J}c(\mathcal F_j)=\bigcap_{j\in J}\mathcal F_j
\end{equation} Combining with the fact that $c$ is extensive, it follows that $\bigcap_{j\in J}
\mathcal F_j$ is a fixed
point of $c$.  On the other hand, let $\{\mathcal F_k\}_{k\in K}$ be a filtered directed family
of objects from $\mathfrak F$ and consider $\mathcal F=\bigcup_{k\in K}\mathcal F_k$. We now
pick some object $f\in \mathcal F$. Then, we can find some $k_0\in K$ such that $f\in \mathcal F_{k_0}$. 
But then, $c(\mathcal K(f))\subseteq c(\mathcal F_{k_0})=\mathcal F_{k_0}$. Since $c$
is of finite type, this now gives $c(\mathcal F)=\bigcup_{f\in \mathcal F}c(\mathcal K(f))\subseteq 
\bigcup_{k\in K}\mathcal F_k=\mathcal F$ and hence $c(\mathcal F)=\mathcal F$. 

\end{proof}

\section{Topological monoids and $\mathcal K$-modules} 

\smallskip
Suppose that $c:SMod(\mathcal M)
\longrightarrow SMod(\mathcal M)$ is an operator that is extensive, order-preserving and of finite type. In this section, our aim is to show that  the spectral space $SMod^c(\mathcal M)$ is a topological
monoid. For this, we will now obtain a more explicit description for the 
smallest thick submodule $\mathcal K(X)\in SMod(\mathcal M)$ containing a generating set $X$ of objects of $\mathcal M$. 
We will do this by extending from \cite[Proposition 3.3]{AB} our methods on the generation of thick tensor ideals. 

\smallskip
Given a set $X$ of objects of $\mathcal M$, we now consider:
\begin{equation}\label{e3.5}
\bar{X}:=\{\mbox{$n\in \mathcal M$ $\vert$ $\exists$ $m\in X$, $a\in \mathcal K$ and $n'\in \mathcal M$ s.t. 
$n\oplus n'\cong a\ast m$  }\}
\end{equation}
We notice that $0\in \bar{X}$ and that $\bar{\bar{X}}=\bar{X}$. On the other hand, we let $\Delta(X)$ denote the collection
of all objects $m\in \mathcal M$ such that there exist $m'$, $m''\in X$ with $m$, $m'$ and $m''$ forming a distinguished
triangle in $\mathcal M$ (in some order). Now if $0\in X$, the fact that $m\overset{1}{\longrightarrow}m\longrightarrow 0$
forms a distinguished triangle for any $m\in \mathcal M$ shows that $X\subseteq \Delta(X)$. 

\smallskip

\begin{thm}\label{P3.5} Let $X$ be a set of objects in $\mathcal M$. Put $X_0:=X$ and inductively define:
\begin{equation}
X_{i+1}:=\Delta(\bar{X}_i)\qquad\forall\textrm{ }i\geq 0
\end{equation} Then, the thick $\mathcal K$-submodule $\mathcal K(X)$ of $\mathcal M$ generated by $X$ is given by the union:
\begin{equation}\label{e3.7}
\mathcal K(X)=\underset{i=0}{\overset{\infty}{\bigcup}}\textrm{ }X_i
\end{equation}
\end{thm}

\begin{proof} From the construction, it is clear that the submodule $\mathcal K(X)$ generated by $X$ contains
each $X_i$. In order to prove \eqref{e3.7}, it therefore suffices to show that 
the union $\underset{i=0}{\overset{\infty}{\bigcup}}\textrm{ }X_i$ is a thick submodule. For the sake of convenience,
we set $X':=\underset{i=0}{\overset{\infty}{\bigcup}}\textrm{ }X_i$. We now choose some $m\in X'$. Then, we can
find some $i\geq 0$ such that $m\in X_i$. Now, for any $a\in \mathcal K$, it is clear that 
$a\ast m\in \bar{X}_i\subseteq \Delta(\bar{X}_i)=X_{i+1}\subseteq X'$. Similarly, if $m$ splits as a direct sum
$m\cong m_1\oplus m_2$, both $m_1 $ and $m_2$ lie in $\bar{X}_i\subseteq X_{i+1}\subseteq X'$. 

\smallskip
Finally, suppose that we have a distinguished triangle $m'\longrightarrow m\longrightarrow m''$ in $\mathcal M$ such 
that two of $m$, $m'$ and $m''\in X'$. For the sake of definiteness, suppose that $m'$, $m''\in X'$. Then we can choose
some $j\geq 0$ large enough so that both $m'$, $m''\in X_j$. But then, $m\in \Delta(X_j)\subseteq \Delta(\bar{X}_j)
=X_{j+1}\subseteq X'$. It follows that $X'$ is a thick $\mathcal K$-submodule of $\mathcal M$.

\end{proof}

\smallskip
We note that one of the simple consequences of Proposition \ref{P3.5} is the fact that if $\{a_i\}_{i\in I}$ is a family
of objects of $\mathcal K$ and $\{m_j\}_{j\in J}$ is a family of objects of $\mathcal M$, the submodule generated
by   $\{a_i\ast m_j\}_{i\in I,j\in J}$ contains all the objects $a\ast m$, where $a$ (resp. $m$) lies in the ideal 
of $\mathcal K$ (resp. the submodule of $\mathcal M$) generated by $\{a_i\}_{i\in I}$  (resp. by $\{m_j\}_{j\in J}$). 
In the case of thick tensor ideals with $\mathcal K=\mathcal M$, we have noted this consequence  in \cite[Lemma 3.4]{AB}.  However, we should
mention that this  fact has been previously established by a different approach (in the slightly different case of localizing submodules)
by Stevenson in \cite[Lemma 3.11]{GS}. 

\smallskip
The following result gives us a better understanding of generating sets of thick $\mathcal K$-submodules: we show that any element $m$ in the submodule
$\mathcal K(X)$ can be obtained starting from only finitely many objects in the generating set $X$.

\smallskip
\begin{thm}\label{P3.6}  Let $X$ be a set of objects in $\mathcal M$ and let $\mathcal K(X)$ be the thick submodule generated by $X$. Then, given an object $m\in \mathcal K(X)$, there exist finitely many objects $m_1$, $m_2$, ... $m_k\in X$ such that
$m$ lies in the submodule generated by the set $\{m_1,m_2,...,m_k\}$.
\end{thm}

\begin{proof} We maintain the notation from the proof of Proposition \ref{P3.5}. We know that 
$\mathcal K(X)=\underset{i=0}{\overset{\infty}{\bigcup}}\textrm{ }X_i$. We now suppose that for any $0\leq j\leq N$ 
and any object $m\in X_j$, we can find finitely many objects $m_1$, $m_2$, ... $m_k\in X$ such that
$m$ lies in the submodule generated by the set $\{m_1,m_2,...,m_k\}$. This is already true for $N=0$. We now pick 
an object $m\in X_{N+1}$. 

\smallskip
By definition, $X_{N+1}=\Delta(\bar{X}_N)$ and hence we can find elements $n'$, $n''\in \bar{X}_N$ such that
$m$, $n'$ and $n''$ form a distinguished triangle (in some order). Now, applying the definition of 
$\bar{X}_N$, it follows that we can find objects $a'$, $a''\in \mathcal K$ and $m'$, $m''\in  X_N$ such that
$n'$ (resp. $n''$) is a direct summand of $a'\ast m'$ (resp. $a''\ast m''$). It follows that $m\in X_{N+1}$ lies
in the submodule generated by $m'$ and $m''$. 

\smallskip
However, since $m'$ and $m''$ lie in $X_N$, we can find a finite set $\{m_1,...,m_k\}$ (resp.  
$\{m_{k+1},...,m_l\}$) of objects in $X$ such that $m'$ (resp. $m''$) lies in the submodule
generated by $\{m_1,...,m_k\}$ (resp.  
$\{m_{k+1},...,m_l\}$). Then, $m\in X_{N+1}$ must lie in the thick $\mathcal K$-submodule generated by  the finite set 
$\{m_1,...,m_k,m_{k+1},...,m_l\}\subseteq X$. This proves the result.

\end{proof}

Given submodules $\mathcal N$, $\mathcal N'\in SMod(\mathcal M)$,
we denote by $\mathcal N+\mathcal N'$ the smallest thick submodule of $\mathcal M$ containing both 
$\mathcal N$ and $\mathcal N'$.  It is clear that addition of submodules makes $SMod(\mathcal M)$ into a commutative monoid.
However, in order to make $SMod^c(\mathcal M)$ into a monoid, we will need the following
result.

\begin{lem}\label{monoid} (a) Let $c:SMod(\mathcal M)\longrightarrow SMod(\mathcal M)$ be an operator
that is extensive, order-preserving and of finite type. Given a submodule $\mathcal N\subseteq 
\mathcal M$, we set $c^\infty(\mathcal N):=\underset{i\geq 0}{\bigcup} c^i(\mathcal N)$. 
Then, for any $\mathcal N\in SMod(\mathcal M)$, the object $c^\infty(\mathcal N)$ lies in $SMod^c(\mathcal M)$. 

\smallskip
(b) The operator $c^\infty: SMod(\mathcal M)\longrightarrow SMod(\mathcal M)$ is of finite type, 
i.e., for any $\mathcal N\in SMod(\mathcal M)$, we have $c^\infty(\mathcal N)=\underset{n\in \mathcal N}{\bigcup}c^\infty(\mathcal K(n))$.
\end{lem}

\begin{proof}  (a) We choose some $\mathcal N\in SMod(\mathcal M)$ and some $n\in c(c^\infty(\mathcal N))$. Since $c$ is of finite type, there
exists some $n_0\in c^\infty(\mathcal N)$ such that $n\in c(\mathcal K(n_0))$. Then, we can choose
$i_0\geq 1$ such that $n_0\in c^{i_0}(\mathcal N)$, i.e., $\mathcal K(n_0)\subseteq c^{i_0}(\mathcal N)$. But then,
$c(\mathcal K(n_0))\subseteq c^{i_0+1}(\mathcal N)\subseteq c^\infty(\mathcal N)$ which shows that
$n\in c^\infty(\mathcal N)$, i.e., $c(c^\infty(\mathcal N))\subseteq c^\infty(\mathcal N)$. Since $c$
is extensive, it follows that $c^\infty(\mathcal N)\in SMod^c(\mathcal N)$. 

\smallskip
(b) For the sake of convenience, we set $\mathcal N':=\underset{n\in \mathcal N}{\bigcup}c^\infty(\mathcal K(n))$. 
For any $n_1$, $n_2\in \mathcal N$, it is clear that $c^\infty(\mathcal K(n_1))$, $c^\infty(\mathcal K(n_2))$ 
both lie inside $c^\infty(\mathcal K(n_1\oplus n_2))$, which shows that $\mathcal N'$ is a filtered 
union of submodules and hence $\mathcal N'\in SMod(\mathcal M)$. We now choose some $n'\in \mathcal N'$. 
Then, $n'\in c^\infty(\mathcal K(n))$ for some $n\in \mathcal N$, i.e., $\mathcal K(n')\subseteq c^\infty(\mathcal K(n))$.
From part (a), we know that $c^\infty(\mathcal K(n))$ is fixed by $c$ and hence
$c(\mathcal K(n'))\subseteq c(c^\infty(\mathcal K(n)))=c^\infty(\mathcal K(n))\subseteq \mathcal N'$. Since
$c$ is an operator of finite type, we know that $c(\mathcal N')=\underset{n'\in \mathcal N'}{\bigcup}c(\mathcal K(n'))$
and hence $c(\mathcal N')=\mathcal N'$. On the other hand, it is clear from the definitions
that $\mathcal N\subseteq \mathcal N'\subseteq c^\infty(\mathcal N)$. It follows that 
$c^i(\mathcal N)\subseteq \mathcal N'\subseteq c^\infty(\mathcal N)$ for each $i\geq 0$. 
Hence, $c^\infty(\mathcal N)=\mathcal N'$, which proves the result. 

\end{proof}

\begin{thm}\label{P3.7} Let $c:SMod(\mathcal M)\longrightarrow SMod(\mathcal M)$ be an operator
that is extensive, order-preserving and of finite type. Let $\mathcal N\in SMod^c(\mathcal M)$ be a submodule
of $\mathcal M$ that is fixed by $c$. Then, the function $f$ defined as follows:
\begin{equation}
f:SMod^c(\mathcal M)\longrightarrow SMod^c(\mathcal M) 
\qquad \mathcal N'\mapsto c^\infty(\mathcal N+\mathcal N')
\end{equation} is a continous function on the spectral space $SMod^c(\mathcal M)$. In other words, 
the spectral space $SMod^c(\mathcal M)$ equipped with the operation $(\mathcal N,
\mathcal N')\mapsto c^\infty(\mathcal N+\mathcal N')$  is a topological
monoid.
\end{thm}

\begin{proof} Since the collection $\{U(m)\cap SMod^c(\mathcal M)\}_{m\in \mathcal M}$ forms a basis of open sets, it suffices
to check that each $f^{-1}(U(m)\cap SMod^c(\mathcal M))$ is open in $SMod^c(\mathcal M)$. For each object $m\in \mathcal M$, we
denote by $\mathcal N^c(m)$ the set of isomorphism classes of objects $m'\in \mathcal M$ such that $m
\in c^\infty(\mathcal N+\mathcal K(m'))$.  We claim that:
\begin{equation}
f^{-1}(U(m)\cap SMod^c(\mathcal M))=\underset{m'\in \mathcal N^c(m)}{\bigcup}\textrm{ }(U(m')
\cap SMod^c(\mathcal M))
\end{equation}
On the one hand, if we have any $\mathcal N'\in U(m')\cap SMod^c(\mathcal M)$ for some $m'\in \mathcal N^c(m)$, then
$m'\in \mathcal N'$ and hence $m\in c^\infty(\mathcal N+\mathcal N')$, i.e., $f(\mathcal N')=c^\infty(\mathcal N+\mathcal N')\in 
U(m)\cap SMod^c(\mathcal M)$. 

\smallskip
Conversely, suppose that we choose some $\mathcal N'\in f^{-1}(U(m)\cap SMod^c(\mathcal M))$. Then, $m\in c^\infty(\mathcal N+
\mathcal N')$. From Lemma \ref{monoid}(b), we know that  $c^\infty$ is of finite type and we can choose $n\in \mathcal N+\mathcal N'$ such that
$m\in c^\infty(\mathcal K(n))$. We let $X$ (resp. $X'$) denote the set of isomorphism classes of objects in 
$\mathcal N$ (resp. $\mathcal N'$). Then, $X\cup X'$ is a generating set for the submodule 
$\mathcal N+\mathcal N'$. We know that $n\in \mathcal N+\mathcal N'$ and it follows from Proposition \ref{P3.6}
that we can choose a finite set  $\{n_1,...,n_k,n_{k+1},...,n_l\}\subseteq X\cup X'$ such that $\{n_1,...,n_k\}\subseteq X$,
$\{n_{k+1}, ..., n_l\}\in X'$ and $n\in \mathcal N+\mathcal N'$ lies in the submodule generated by
$\{n_1,...,n_k,n_{k+1},...,n_l\}$. We now consider the object $n_0:=n_{k+1}\oplus n_{k+2}\oplus ...\oplus n_l$ which lies
in $\mathcal N'$, i.e., $\mathcal N'\in U(n_0)$. Also $m$ lies in $c^\infty(\mathcal N+\mathcal K(n_0))$
 and hence we can find some $n'\in \mathcal N^c(m)$ such that $n_0\cong n'$. It follows that 
$\mathcal N'\in U(n_0)=U(n')$. Hence, the result follows.

\end{proof}


\small
\begin{thebibliography}{99}


\bibitem{BMain} P.~Balmer,
The spectrum of prime ideals in tensor triangulated categories, 
{\it J. Reine Angew. Math.} {\bf 588} (2005), 149--168. 

\bibitem{Balmer2} P.~Balmer, Supports and filtrations in algebraic geometry and modular representation theory,
{\it Amer. J.
Math.} {\bf 129} (5), 1227--1250 (2007). 

\bibitem{Balmer3} P.~Balmer, Spectra, spectra, spectra -- tensor triangular spectra versus Zariski spectra of endomorphism
rings, {\it Algebr. Geom. Topol.} {\bf 10}(3), 1521--1563 (2010). 

\bibitem{Balmer4} P.~Balmer, Tensor triangular geometry. In: Bhatia, R. (ed.) {\it Proceedings of the International Congress
of Mathematicians}. Vol. II, pp. 85--112, Hindustan Book Agency, New Delhi (2010).

\bibitem{BF1} P.~Balmer, G.~Favi,  Gluing techniques in triangular geometry, {\it Q. J. Math.} {\bf  58} (4), 415--441 (2007).

\bibitem{BF2} P.~Balmer, G.~Favi,  Generalized tensor idempotents and the telescope conjecture, 
{\it  Proc. London Math.
Soc.3) } {\bf 102} (6), 1161--1185 (2011)

\bibitem{AB} A.~Banerjee,  Realizations of pairs and Oka families in tensor triangulated categories, {\it European Journal of Mathematics}, Vol. {
\bf 2}, No. 3, 760--797 (2016).

\bibitem{AB2} A.~Banerjee, Closure operators in abelian categories and spectral spaces, Preprint (2016). 

\bibitem{BCR} D.~J.~Benson, J.~F.~Carlson, J.~Rickard,  Thick subcategories of the stable module category,
{\it  Fund. Math.}
{\bf 153}, 59--80 (1997).

\bibitem{DHS} E.~S.~Devinatz, M.~J.~Hopkins, J.~H.~Smith,  Nilpotence and stable homotopy theory I,
{\it  Ann. Math. (2)}
{\bf 128} (2), 207--241 (1988). 

\bibitem{Eps} N.~Epstein, 
A guide to closure operations in commutative algebra, {\it Progress in commutative algebra 2}, 1--37, Walter de Gruyter, Berlin, 2012.


\bibitem{Eps1} N.~Epstein, 
Semistar operations and standard closure operations, 
{\it Comm. Algebra} {\bf 43} (2015), no. 1, 325--336. 

\bibitem{Fino2} C.~A.~Finocchiaro, M.~Fontana, K.~A.~Loper, Ultrafilter and constructible topologies on spaces of valuation domains, {\it Comm. Algebra},  {\bf 41} (2013), no. 5, 1825--1835.

\bibitem{Fino} C.~A.~Finocchiaro,  Spectral spaces and ultrafilters, {\it Comm. Algebra}
{\bf  42} (2014), no. 4, 1496--1508.

\bibitem{FiFo} C.~A.~Finocchiaro, M.~Fontana, D.~Spirito, A topological version of Hilbert's Nullstellensatz,
{\it J. Algebra}, {\bf 461} (2016), 25--41. 


\bibitem{Hoch} M.~Hochster, Prime ideal structure in commutative rings, {\it Trans. Amer.
Math. Soc.} {\bf 142} (1969), 43--60.

\bibitem{Kl} S.~Klein,  Chow groups of tensor triangulated categories, {\it J. Pure Appl. Algebra}, {\bf 220} (4), 1343--1381
(2016). 

\bibitem{Kl2} S.~Klein, Intersection products for tensor triangular Chow groups,  {\it J. Algebra} {\bf 449}, 497--538 (2016). 

\bibitem{Krause} H.~Krause, 
Deriving Auslander's formula,  
{\it Doc. Math.} {\bf 20} (2015), 669--688.

\bibitem{Pete} T.~J.~Peter, Prime ideals of mixed Artin–Tate motives, {\it J. K-Theory} {\bf 11} (2), 331--349 (2013).

\bibitem{Bernie} B.~Sanders, Higher comparison maps for the spectrum of a tensor triangulated category, {\it Adv. Math.}
{\bf 247}, 71--102 (2013).

\bibitem{GS} G.~Stevenson,  Support theory via actions of tensor triangulated categories,
{\it J. Reine Angew. Math.} {\bf 681} (2013), 219--254. 

\bibitem{GS2} G.~Stevenson, Subcategories of singularity categories via tensor actions, 
{\it  Compositio Math. } {\bf 150} (2),
229--272 (2014).

\bibitem{GS4} G.~Stevenson, Derived categories of absolutely flat rings, 
{\it Homology Homotopy Appl.} {\bf 16} (2014), no. 2, 45--64. 



\bibitem{Thom} R.W.~Thomason, The classification of triangulated subcategories, {\it Compositio Math.}
{\bf  105}, 1--27 (1997).

\end{thebibliography}
\end{document}